\documentclass[12pt,a4paper]{article}

\usepackage[T1]{fontenc}
\usepackage{authblk}

\usepackage[utf8]{inputenc}
\usepackage[english]{babel}
\usepackage{amsmath}
\usepackage{amsfonts}
\usepackage{amssymb}
\usepackage[left=2.5cm,right=2.5cm,top=2.5cm,bottom=2.5cm]{geometry}
\usepackage{algorithm}
\usepackage{algpseudocode}
\usepackage{graphicx}
\usepackage{caption}
\usepackage{subcaption}
\usepackage{multicol}
\usepackage{mathtools}

\usepackage{amssymb}
\usepackage{amsthm}

\usepackage{multicol}
\usepackage{color}

\newtheorem{thm}{Theorem}[section]
\newtheorem{definition}[thm]{Definition}

\newtheorem{prop}[thm]{Proposition}

\newcommand{\chg}[1]
{
\textcolor{black}{#1}
}

\makeatletter
\def\BState{\State\hskip-\ALG@thistlm}
\makeatother

\author[1]{Bernadett \'Acs}
\author[1,2]{G\'abor Szederk\'enyi}
\author[3,4]{Zsolt Tuza} 
\author[5]{Zolt\'an A. Tuza\footnote{This research was performed while ZT was with the Faculty of Information Technology and Bionics of Pázmány Péter Catholic University.}}

\affil[1]{\small P\'azm\'any P\'eter Catholic University, Faculty of Information Technology and Bionics, \mbox{Práter u. 50/a},  H-1083 Budapest, Hungary} 
\affil[2]{\small Systems and Control Laboratory, Institute for Computer Science and Control (MTA SZTAKI) of the Hungarian Academy of Sciences, Kende u. 13-17, H-1111 Budapest, Hungary}
\affil[3]{\small Department of Computer Science and Systems Technology, University of Pannonia, \mbox{Egyetem u. 10, H-8200 Veszprém, Hungary}}
\affil[4]{\small Alfréd Rényi Institute of Mathematics of the Hungarian Academy of Sciences, \mbox{Reáltanoda u. 13-15, H-1053 Budapest, Hungary}}
\affil[5]{\small University of Stuttgart Institute for Systems Theory and Automatic Control, \mbox{Pfaffenwaldring 9,  D-70569 Stuttgart, Germany}}

\affil[ ]{\small e-mail: acs.bernadett@itk.ppke.hu, szederkenyi@itk.ppke.hu, tuza@dcs.uni-pannon.hu, zoltan.tuza@ist.uni-stuttgart.de}
\title{Computing all possible graph structures describing linearly conjugate realizations of kinetic systems}
\date{}

\begin{document}

\maketitle

\begin{abstract}
In this paper an algorithm is given to determine all possible structurally different linearly conjugate 
realizations of a given kinetic polynomial system. The solution is based on the iterative search for 
constrained dense realizations using linear programming. Since there might exist exponentially many different 
reaction graph structures, we cannot expect to have a polynomial-time algorithm, but we can organize the computation in such a way that 
polynomial time is elapsed between displaying any two consecutive realizations. The correctness of the 
algorithm is proved, and \chg{possibilities of a parallel implementation} are discussed. The operation of the 
method is shown on two illustrative examples.
\end{abstract}


\section{Introduction}
Chemical reaction networks (CRNs) obeying the mass action law can be originated  
from the dynamical modelling of chemical and biochemical processes. However, these models are also capable of
describing all important phenomena of nonlinear dynamical behaviour and have several applications in different 
fields of science and engineering \cite{Erdi1989, Samardzija1989}. 
Kinetic systems, which describe the dynamical behaviour of the CRNs 
have a simple algebraic characterization, which makes it possible to develop effective computational methods 
for dynamical analysis, and even control \cite{Angeli2009,Chellaboina2009,Haddad2010}. 

The graph representation of reaction networks, especially the exploration 
of the relation between the reaction graph structure and the dynamics of the network (preferably without the 
precise knowledge of the reaction rate coefficients) has become an important research area in chemical 
reaction network theory since the 1970's. \chg{One of the first comprehensive overviews of mass-action-type CRNs can be found in \cite{Horn1972} clearly indicating that the class of such kinetic systems is far more general than the description of ``chemically reacting mixtures in closed vessels''. In the same article, the Kirchhoff matrix based description of reaction networks is also introduced.} 

\chg{Complex 
balance is a fundamentally important property of CRNs (see, e.g \cite{Horn1972a,Feinberg1972,Anderson2011}). 
Roughly speaking, complex balance means that the signed sum of the incoming and outgoing fluxes corresponding 
to each vertex (complex) at equilibrium is zero. It is an interesting and important property that the set of equilibrium points of any complex balanced system forms a toric variety in the state space \cite{Craciun2009}.  One of the most significant achievements in chemical reaction network theory is the Deficiency Zero Theorem saying that weakly reversible reaction networks of deficiency zero are complex balanced, independently of the values of the reaction rate coefficients \cite{Feinberg1987}. An important 
recent result related to complex balance is the possible general proof of the Global Attractor Conjecture 
\cite{Craciun2015}. According to this conjecture, complex balanced CRNs with mass action kinetics are globally 
stable within the positive orthant with a known logarithmic Lyapunov function that is independent of the 
reaction rate coefficients. The experimentally observed robust dynamical behavior of certain reaction mechanisms can also be understood from the special properties of the network structure \cite{Shinar2010}. }

\chg{It is known from the ``fundamental dogma of chemical kinetics'' that the reaction graph corresponding to a given kinetic dynamics is generally 
non-unique \cite{Espenson1995,Epstein1998a}. This phenomenon is also called \textbf{dynamical equivalence} or \textbf{macro-equivalence} \cite{Horn1972, Hars1981}, and it has been studied in the case of Michaelis-Menten kinetics, too \cite{Schnell2006}.  Clearly, this kind of structural non-uniqueness 
may seriously hamper the parameter estimation (inference) of reaction networks from measurement data, that is 
generally a challenging task \cite{Binder2015,Villaverde2014}. To improve the solvability of such inference problems, it is worth involving as much prior information as possible into the computation problem in the form of additional constraints \cite{Jaynes1984,Szederkenyi2011c,Siegenthaler2014}. 
Necessary and sufficient conditions for a general polynomial vector field to be kinetic were first given in \cite{Hars1981}. In the constructive proof, a procedure was given to construct one possible dynamically equivalent reaction network structure (called the canonical structure) realizing a kinetic dynamics. 
The first theoretical results about dynamical 
equivalence of CRNs were published in \cite{Craciun2008} pointing out the convex cone structure of the parameter 
space. Motivated by these results, the numerical computation of dynamically equivalent realizations was first 
put into an optimization framework in \cite{Szederkenyi2009b} by proposing a method for determining the graph 
structures containing the minimum or maximum number of reactions. It is also important to know that several 
fundamental properties of CRNs such as (weak) reversibility, complex and detailed balance or deficiency  
are not `encoded' into the kinetic differential equations but may vary with the dynamically equivalent 
structures corresponding to the same dynamics. This fact is an important motivating factor to develop methods for the computation of reaction graphs of kinetic systems.} 

There is a significant extension of dynamical equivalence 
called linear conjugacy, where the kinetic model is subject to a positive definite linear diagonal 
state transformation \cite{Johnston2011conj}. 
Obviously, linear conjugacy preserves the main qualitative dynamical properties of CRNs like stability and
multiplicities or the boundedness of solutions. However, due to the larger degree of freedom introduced by the 
transformation parameters, it allows a wider variety of possible reaction graph structures than dynamical 
equivalence does. \chg{Therefore, several optimization based computational methods have been suggested to find 
linearly conjugate, or as a special case, dynamically equivalent realizations of kinetic systems having 
preferred properties such as density/sparsity, complex or detailed balance, or minimum deficiency 
\cite{Johnston2013}. An algorithm for finding all possible sparse dynamically equivalent reaction network 
structures was proposed and applied for a Lorenz system transformed into kinetic form in \cite{LorenzCRN}, 
using the mixed integer linear programming (MILP) framework. }

After treating the above mentioned important special cases,  a question arises naturally: Is it 
possible to give a computationally efficient algorithm for  determining all possible  reaction graph 
structures corresponding to linearly conjugate CRN realizations of a given kinetic 
system? The aim of this 
paper is to give a solution to this problem. Of course, we cannot expect to find a polynomial-time algorithm 
for the overall problem, since exponentially many different realizing graph structures may exist for a kinetic 
model. However, as it will be shown, it is possible to achieve that each computation step (i.e. finding and 
displaying the next realization after the previous one) is performed in polynomial time. 

\section{Basic notions}
In this section, we summarize the basic notions and results for both algebraic and graph-based representations 
of CRNs. We will use the following notations:

\begin{tabular}{ll}
$\mathbb{R}$ & the set of real numbers\\
$\mathbb{R_+}$ & the set of nonnegative real numbers\\
$\mathbb{N}$ & the set of natural numbers, including 0\\
$H^{n \times m}$ & the set of matrices having entries from a set $H$ in  $n$ rows and $m$ columns\\
$[M]_{ij}$ & the entry in row $i$ and column $j$ of matrix $M$
\end{tabular}

\subsection{Algebraic characterization} \label{subs:alg}
It is important to remark here that similarly to \cite{Feinberg1987}, \cite{Erdi1989} or 
\cite{Samardzija1989}, we consider reaction networks as a general system class representing nonlinear 
dynamical systems with nonnegative states, therefore we do not require that they fulfil actual physico-
chemical constraints like mass conservation. In other words, there is no constraint on how different 
complexes can transform into each other in the network.

\begin{definition}
Chemical reaction networks can be determined by the following three sets \emph{(see e.g. 
\cite{Feinberg:79, Feinberg1987})}.
\begin{itemize}
\item A set of \textbf{species}: $\mathcal{S}= \{X_i\  |\  i \in \{1,\ldots,n\}\}$
\item A set of \textbf{complexes}: $\mathcal{C} = \{ C_j \ |\  j\in\{1,\ldots,m\}\}$, where\\
$C_j = \sum \limits_{i=1}^{n} \alpha_{ji} X_i \qquad j \in \{1,\ldots,m\}$  \\
$\alpha_{ji} \in \mathbb{N} \qquad \qquad \quad j \in\{1,\ldots,m\}, \   i \in \{1,\ldots,n\} $ 

The complexes are formal linear combinations of the species with coefficients, called the 
\textbf{stoichiometric coefficients}.
\item A set of \textbf{reactions}: $\mathcal{R} \subseteq \{(C_i,C_j)\  |\  C_i,C_j \in \mathcal{C}\}$\\
The ordered pair $(C_i,C_j)$ corresponds to the reaction $C_i \rightarrow C_j$.
\end{itemize}

\noindent To each ordered pair  $(C_i,C_j)$ where $i,j \in \{1, \ldots m\}$ \chg{and $i \neq j$}  there 
belongs a nonnegative real number $k_{ij}$ called the \textbf{reaction rate coefficient}. The reaction 
$C_i \rightarrow C_j$ takes place if and only if the corresponding coefficient is positive.
\end{definition}

\bigskip

\noindent The properties of the reaction network are encoded by special matrices.

\begin{definition} \label{Y}
$Y \in \mathbb{N}^{n \times m}$ is the \textbf{complex composition matrix} of the CRN if its entries are the 
stoichiometric coefficients.
\begin{equation}
[Y]_{ij}=\alpha_{ji} \qquad  i \in \{1,\ldots,n\}, \  j \in\{1,\ldots,m\}
\end{equation}
\end{definition}

\begin{definition} \label{Ak}
$A_k \in \mathbb{R}^{m \times m}$ is the \textbf{Kirchhoff matrix} of the CRN if its entries are determined by 
the reaction rate coefficients as follows:
\begin{equation}
[A_k]_{ij}= \begin{cases} 
k_{ji} &\text{ if } i \neq j\\
-\sum \limits_{l=1, l\neq i}^{m} k_{il} & \text{ if } i=j\\
\end{cases}
\end{equation}
Since the sum of entries in each column is zero, $A_k$ is also called a 
\textbf{column conservation matrix}.
\end{definition}

Let $x:\mathbb{R} \rightarrow \mathbb{R}^n_{+}$ be a function, which defines the  
\textbf{concentrations of the species} depending on time. 
Assuming mass-action kinetics, the dynamics of the function $x$ can be described by dynamical 
equations of the following form:

\begin{equation}\label{eq:kinsys}
\dot{x}=Y \cdot A_k \cdot \psi(x)
\end{equation}
where $\psi: \mathbb{R}_+^n \rightarrow \mathbb{R}_+^m$  is a monomial-type vector-mapping,
\begin{equation} \label{eq:psi}
\psi_j(x)=\prod \limits_{i=1}^{n} x_i^{\alpha_{ji}}, \quad j =1,\dots,m
\end{equation}

\bigskip

According to Definitions \ref{Y}, \ref{Ak} and Equations \eqref{eq:kinsys}, \eqref{eq:psi} the dynamics of a 
reaction network can be given by a set of polynomial ODEs, but not every polynomial system describes a CRN.

\begin{definition}
Let $x: \mathbb{R} \rightarrow \mathbb{R}_{+}^n$ be a function, $M\in\mathbb{R}^{n\times p}$ a matrix and 
$\varphi: \mathbb{R}_+^n \rightarrow \mathbb{R}_+^p$ a monomial function.
The polynomial system
\begin{equation}\label{eq:kin_poly}
\dot{x} = M \cdot \varphi(x)
\end{equation}
is called \textbf{kinetic} if there exist a matrix  $Y\in\mathbb{N}^{n\times m}$ and a Kirchhoff matrix 
$A_k\in\mathbb{R}^{m\times m}$, so that
\begin{equation} \label{eq:dyneq}
M\cdot \varphi(x) = Y\cdot A_k\cdot \psi(x)
\end{equation}
where $\psi:\mathbb{R}_+^n \rightarrow \mathbb{R}_+^m$ is a monomial function determined by the entries of  
matrix $Y$, $\psi_j(x)=\prod_{i=1}^n x_i^{[Y]_{ij}}$ for $j \in \{1,\dots,m \}$.
\end{definition}

The matrices $Y$ and $A_k$ characterize the dynamics of the kinetic system, as well as the CRN. However, the 
polynomial kinetic system does not uniquely determine the matrices. Reaction networks with different sets of 
complexes and reactions can be governed by the same dynamics, see e.g. 
\cite{Horn1972, Craciun2008, Szederkenyi2009b}. 
If the matrices $Y$ and $A_k$ of a reaction network fulfil Equation \eqref{eq:dyneq}, then the CRN is 
called a \textbf{dynamically equivalent realization} of the kinetic system \eqref{eq:kinsys}, and it is 
denoted by the matrix pair $(Y,A_k)$.

\bigskip

The notion of dynamical equivalence can be extended to the case when the polynomial system is 
subject to a positive linear diagonal state transformation. It is known 
from \cite{Farkas1999} that such a transformation preserves the kinetic property of the system.

Let $T \in \mathbb{R}^{n \times n}$ be a positive definite diagonal matrix. The state 
transformation is performed as follows:
\begin{equation}
x=T \cdot \bar{x}, \quad \bar{x}=T^{-1} \cdot x
\end{equation}
Applying it to the polynomial system  \eqref{eq:kin_poly} we get
\begin{equation}
\dot{\bar{x}}=T^{-1} \cdot \dot{x}=T^{-1} \cdot  M \cdot \varphi(x)=T^{-1} \cdot  M \cdot \varphi(T \cdot \bar{x})=T^{-1}\cdot M \cdot \Phi_T \cdot \varphi(\bar{x})
\end{equation}
where $\Phi_T \in \mathbb{R}^{n \times n}$ is a positive definite diagonal matrix so that 
$[\Phi_T]_{ii} = \varphi_i(T \cdot \mathbf{1})$ for $i \in \{1, \ldots, n\}$, $\varphi_i$ is the 
$i$th coordinate function of $\varphi$, and $\mathbf{1} \in \mathbb{R}^n$ is a column vector with all 
coordinates equal to $1$. Now we can give the extended definition.
\begin{definition}\label{def:linconj}
A reaction network $(Y,A_k')$ is a \textbf{linearly conjugate realization} of the kinetic system 
\eqref{eq:kin_poly}  if there exists a positive definite diagonal matrix $T \in \mathbb{R}^{n\times n}$ so 
that
\begin{equation}\label{eq:linconj}
Y\cdot A_k'\cdot \psi(x)  = T^{-1}\cdot M \cdot \Phi_T \cdot \varphi(x)
\end{equation}
where $Y\in\mathbb{N}^{n\times m}$, $\psi: \mathbb{R}_+^n \rightarrow \mathbb{R}_+^m$ with 
$\psi_j(x)=\prod_{i=1}^n x_i^{[Y]_{ij}}$ for $j\in \{ 1,\ldots,m\}$,
and  $A_k' \in\mathbb{R}^{m\times m}$ is a Kirchhoff matrix.
\end{definition}

\bigskip

It can be seen that dynamical equivalence is a special case of linear conjugacy, when the matrix $T$, and 
therefore the matrices $T^{-1}$ and $\Phi_T$ as well are identity matrices.

\bigskip

Since the monomial functions $\varphi$ and $\psi$ in Equations \eqref{eq:dyneq} and \eqref{eq:linconj} 
might be different, in both cases the set of complexes is not fixed. 
By applying the method described in \cite{Hars1981} a suitable set of complexes can be determined, but there 
are other possible sets as well.
It is clear that the complexes determined by the monomials of function $\varphi$ must be in the set 
$\mathcal{C}$, but arbitrary further complexes might be involved as well, which appear in the original kinetic 
equations with zero coefficients. 
These additional complexes change the dimensions of the matrices $Y$ and $A'_k$, therefore we have to modify 
the matrices $M$ and $\Phi_T$ as well, in order to get the following equation:

\begin{equation} \label{eq:linconj_mod}
Y\cdot A_k'\cdot \psi(x)  = T^{-1}\cdot M' \cdot \Phi'_T \cdot \psi(x)
\end{equation}

\noindent where the  matrices $M' \in \mathbb{R}^{n \times m}$ and $\Phi'_T \in \mathbb{R}^{m \times m}$ have 
the same columns and diagonal entries as $M$ and $\Phi_T$ belonging to the complexes determined by 
$\varphi$, and zero columns and $1$ diagonal entries belonging to all additional complexes, respectively.

\chg{Since there is the same monomial-type vector-mapping $\psi$ on both sides of Equation \eqref{eq:linconj_mod}, the equation can be fulfilled if and only if the coefficients belonging the same monomials are pairwise identical. This means that by using the notation $A_k=A'_k\cdot {\Phi'_T}^{-1}$, we can rewrite Equation \eqref{eq:linconj_mod} as}

\begin{equation} \label{eq:linconj_simple}
Y\cdot A_k  = T^{-1}\cdot M' 
\end{equation}

\noindent where $A_k$ is a Kirchhoff matrix, too, obtained by scaling the columns of $A_k'$ by positive 
constants. It is easy to see that this operation preserves the set of reactions, but changes 
the values of the non-zero entries. The actual reaction rate coefficients of the linearly conjugate network 
are contained in the matrix 
\begin{align}\label{eq:linconjtraf}
A'_k=A_k \cdot \Phi'_T
\end{align}

From now on we will consider only linearly conjugate realizations on a fixed set of complexes.
A reaction network which is a linearly conjugate realization of a kinetic system can be identified by its 
matrices $T$, $Y$, and $A'_k$. However, since matrix $Y$ is fixed,  and the matrix $A_k$ is returned by the 
computation, we will simply denote this realization by the 
matrix pair $(T,A_k)$. 

\subsection{Graph representation}\label{subs:graph}
A reaction network can be described by a weighted directed graph.

\begin{definition}
The graph $G(V,E)$ representing the CRN is called \textbf{Feinberg--Horn--Jackson graph}, or 
\textbf{reaction graph} for short. If the weights are given by the function 
$w: E(G) \rightarrow \mathbb{R}_{+}$, the reaction graph is defined as follows:
\begin{itemize}
\item the \textbf{vertices} correspond to the complexes, $V(G)=\mathcal{C}$,
\item the \textbf{directed edges} describe the reactions, $E(G)=\mathcal{R}$,\\ 
there is a directed edge from vertex $C_i$ to vertex $C_j$  if and only if the reaction $C_i \rightarrow C_j$ 
takes place (i.e. $k_{ij}>0$),
\item the \textbf{weights} of the edges are the reaction rate coefficients, \\
$w((C_i,C_j))= k_{ij}$ where $(C_i,C_j) \in \mathcal{R}$.
\end{itemize}
Loops and multiple edges are not allowed in a reaction graph.
\end{definition}

The applicability of the algorithm presented in this paper depends on a recently proved important property of 
the so-called dense realizations. Therefore, we formally define the notion of dense realizations, and then 
recall the related result from \cite{Acs2015}.
\begin{definition} \label{def_dense}
A realization of a CRN is a \textbf{dense realization} if the maximum number of reactions take place. 
\end{definition}
It is easy to see that the dense property of a CRN realization is equivalent to the feature that its Kirchhoff 
matrix $A_k$  has the maximum number of positive off-diagonal entries.
\chg{In a set of weighted directed graphs we call a graph super-structure if it contains each element as 
a subgraph not considering edge weights, and it is minimal under inclusion. It is clear that the structure of the super-structure graph is 
unique, because there cannot be two different graphs that are subgraphs of each other, but have different 
structures. The following proposition published in \cite{Acs2015} says that dense linearly conjugate realizations 
have a super-structure property even if an arbitrary additional finite set of linear constraints on the rate 
coefficients and the transformation parameters has to be fulfilled.}

\begin{prop} \label{superstr} \emph{\cite{Acs2015}} Among all the realizations linearly conjugate to a given 
kinetic system and fulfilling a finite set of additional linear constraints there is a realization determining 
a super-structure.
\end{prop}

\chg{An important example of the linear constraints mentioned in Proposition \ref{superstr} is the condition 
for mass conservation. According to this, the total mass in the CRN is preserved (i.e. the 
system is called \textit{kinetically mass conserving} \cite{Nagy2014}) if and only if there exists a strictly 
positive vector $k\in\mathbb{R}^n$  such that the following holds:
\begin{equation}\label{eq:masscons}
k^{\top} \cdot Y \cdot  A_k = \mathbf{0}^{\top},
\end{equation}
where $\mathbf{0} \in \mathbb{R}^n$ is the null vector.
The chemical meaning of the vector $k$ is that its entries are the molecular 
(or atomic) weights of the species in the network. By using Equation \eqref{eq:linconjtraf}  
it is easy to see that for a given vector $k$, Equation \eqref{eq:masscons} holds if and 
only if $(k^{\top} \cdot Y \cdot A'_k)^{\top}$ is the null vector, since $\Phi'_T$ is an 
invertible diagonal matrix. Therefore, the computation of all reaction graph structures that describe realizations obeying mass conservation with a given vector $k$ is also possible  by using the 
method proposed in Section \ref{sec_alg} and adding Equation \eqref{eq:masscons} to the computation 
constraints.}

\section{Computation model for computing linearly conjugate realizations} \label{subs:opt_model}
Linearly conjugate realizations can be computed by using a linear optimization model \cite{Johnston2013}. As 
it was described in Section \ref{subs:alg}, Equation \eqref{eq:linconj_simple} must be fulfilled. \chg{We 
assume that the set of complexes is given, so the coefficient matrix does not need any modification, therefore 
$M = M'$ holds.} Consequently the equation characterizing linearly conjugate realizations 
can be written as follows:
\begin{equation} \label{a}
T^{-1} \cdot M - Y \cdot A_k = \mathbf{0}
\end{equation}
where $\mathbf{0} \in \mathbb{R}^{n \times m}$  denotes the zero matrix. The matrices $Y$ and $M$  
are fixed, and the variables are contained in the matrices $T^{-1}$ and $A_k$. Specifically, the 
variables are the off-diagonal entries of the matrix $A_k$  and the diagonal entries of the matrix $T^{-1}$. 
We do not consider the diagonal entries of matrix $A_k$ as variables, 
because these are uniquely determined by the off-diagonal entries of matrix $A_k$:
\begin{equation} \label{b}
[A_k]_{ii} =- \sum \limits_{\substack{j =1 \\ j \neq i}}^{m} [A_k]_{ji} \quad   i \in \{1, \ldots, m\}  
\end{equation}
Equation \eqref{a} guarantees linear conjugacy, and  Equations \eqref{b}, \eqref{c}, and 
\eqref{d} ensure that the matrices $T^{-1}$ and $A_k$ meet their definitions.
\begin{align} 
\ &[A_k]_{ij} \geq 0  &    i,j \in  \{1,\ldots ,m\}, \  i\neq j \label{c} \\
\ &[T^{-1}]_{ii} > 0 &  i \in \{1,\ldots, n\} \label{d}
\end{align}
In our algorithm it will be necessary to exclude some set $\mathcal{H} \subset \mathcal{R}$ of reactions from 
the computed realizations. This can be written in the form of a linear constraint as follows:
\begin{equation} \label{e}
[A_k]_{ji} = 0 \qquad  (C_i,C_j) \in \mathcal{H}
\end{equation}
The feasibility of the linear constraints \eqref{a}--\eqref{e} can be
verified, and valid solutions (if exist) can be determined in the 
framework of linear programming.

It is important to remark  that in the above computation model, the boundedness property of all variables can 
be ensured so that the set of possible reaction graph \chg{structures} remains the same as it was proved in \cite{Acs2015}.
\begin{prop} \emph{\cite{Acs2015}} \label{prop:upperbound}
For any linearly conjugate realization $(T,A_k)$ of a kinetic system there is another linearly conjugate 
realization $(T^*,A^*_k)$ with all variables smaller than the given upper bound(s) so that \chg{the reaction graphs describing the two realizations are structurally identical.}
\end{prop}
In our algorithm it is necessary to compute constrained dense linearly conjugate realizations, which rises two 
technical problems. The first is that we have to maximize the number of positive off-diagonal entries of 
matrix $A_k$. The second one is that there are strict inequalities (the diagonal entries of matrix $T^{-1}$ 
must be strictly positive). There are several possibilities for handling these issues. In our implementation, we will apply the 
method presented in \cite{Acs2015}, which can determine constrained dense linearly conjugate realizations in 
polynomial time. \chg{The reasons for this choice are briefly the following. Several solutions use mixed 
integer linear programming (MILP) for determining dense realizations (see e.g. \cite{Johnston2013} and the 
references therein). However, MILP problems are known to be NP-hard and their application would not allow  
the time complexity of the computation between displaying any two consecutive realizations to be polynomial. 
Moreover, the treatment of strict inequalities in optimization is not trivial and needs special attention 
\cite{Boyd2004}. Therefore, inequalities like $x>c$ are  often transformed in practice to the form 
$x \geq c+ \varepsilon$, where  $\varepsilon >0$ is a small number, but this may make the computations less 
accurate. The method in \cite{Acs2015} avoids both issues mentioned above by utilizing the special properties 
of linearly conjugate realizations, and therefore, this is the most reliable polynomial-time method that we 
currently know.}

\section{Algorithm for computing all realizations} \label{sec_alg}

\chg{Linearly conjugate realizations are parametrically not unique, the matrices $A_k$ and $T^{-1}$ in 
Equation \eqref{a} can be scaled by the same arbitrary positive scalar \cite{Johnston2011conj,Acs2015}. 
Therefore, our aim is to determine all the possible reaction graph structures describing linearly conjugate 
realizations of a kinetic system. In this section we present and analyze an algorithm for computing all such 
reaction graph structures. This is the main result of the paper.}

\bigskip

\chg{From now on the reaction graphs will be assumed to be unweighted directed graphs, i.e. by reaction graph 
we will mean only its structure.}
According to Proposition \ref{superstr}, if $G_D$ is the reaction graph of the dense realization and graph 
$G_R$ describes another realization, then $E(G_R) \subseteq E(G_D)$ holds. As it was proven in \cite{Acs2015}, 
the dense realization can be determined by a polynomial algorithm, which is the first step of our method.

There might be reactions that take place in each realization. These are called \textbf{core reactions}, and 
the edges representing them are the \textbf{core edges}. The set of core edges is denoted by $E_c$, which 
can also be determined by a polynomial algorithm \cite{Szederkenyi2011c}. \chg{It is worth doing this step, 
since it might save some computational time and space, but it is not necessary for the running of the 
algorithm.}

Based on the above, each reaction graph can be uniquely determined if it is known which non-core reactions of 
the dense realization take place in the reaction network. \chg{Thus,} we represent the reaction 
graphs by  binary sequences of length $N = | E(G_D)\setminus E_c|$.

In order to define the binary sequences we fix an ordering of the non-core edges. Let $e_i$ denote the 
$i$th edge.  If $R$ is a binary sequence, then let $R[i]$ denote the $i$th coordinate of the sequence and 
$G_R$ be the reaction graph described by $R$. If a realization can be decoded by the binary sequence $R$, then
\begin{equation}
e \in E(G_R) \Longleftrightarrow
\begin{cases}
e \in E_c \\
\quad or \\
\exists i \in \{1, \ldots , N\} \quad e=e_i , \  R[i]=1
\end{cases}
\end{equation}

From now on the term `sequence' will refer to such a binary sequence of length $N$. The sequence representing 
the dense realization (with all coordinates equal to 1) will be denoted by $D$.

For the efficient operation of the algorithm, appropriate data structures are needed as well.
The discovered graph structures are stored in a binary array of size $2^N$ called $Exist$, where the 
indices of the fields are the sequences as binary numbers. At the beginning the values 
in every field are zero, and after the computation the value of field $Exist[R]$ is 1 if and only if there is 
a linearly conjugate realization described by the sequence $R$.

We also need $N+1$ stacks, indexed from $0$ to $N$. The $k$th stack is referred to 
as $S(k)$. During the computation, sequences are temporarily stored in these stacks, following the 
rule: the sequence $R$ describing a linearly conjugate realization might be in stack $S(k)$ if and only if 
there are exactly $k$ coordinates of $R$ which are equal to 1, i.e. exactly $k$ reactions take 
place in the realization.

At the beginning all stacks are empty, but during the running of the algorithm we push in and pop out 
sequences from them. The command `push $R$ into $S(k)$'  pushes the sequence $R$ into the stack $S(k)$, and 
the command `pop $S(k)$' pops a sequence out from $S(k)$ and returns it. 
(It makes no difference, in what order are the elements of the stacks popped out, but by the definition of the 
data structure the sequence pushed in last will be popped out first.)
The number of sequences in stack $S(k)$ are denoted by size.$S(k)$, and the number of coordinates equal to 
1 in the sequence $R$ are referred to as $e(R)$.

\bigskip

\noindent Within the algorithm the following procedure is used repeatedly:

\bigskip

\noindent \textbf{FindLinConjWithoutEdge}($M,Y,R,i$) computes a constrained dense linearly conjugate 
realization of the kinetic system with coefficient matrix $M$ and complex composition matrix $Y$. The 
additional inputs $R$ and $i$ are a  sequence encoding the input reaction graph structure, and  an integer 
index, respectively. The procedure returns a  sequence $U$ encoding the graph structure of the computed 
realization such that $G_U$ is a subgraph of $G_R$ and $U[i]=0$. \chg{If there is no such realization, then 
$-1$ is returned.} This computation can be carried out in polynomial time as it is described in 
\cite{Acs2015}.

\bigskip

\begin{algorithm}[H]
\caption{Determines all reaction graphs describing linearly conjugate realizations}
\label{alg:all}
\begin{algorithmic}[1]
\Procedure{Linearly conjugate graph structures}{$M,Y$}
\State push $D$ into $S(N)$ 
\State $Exist[D]$:=1
\For {$k=N$ to $1$}
\While{size.$S(k)>0$}
\State $R$:= pop $S(k)$
\For {$i = 1$ to $N$}
\If {$R[i]=1$}
\State $U:=$ FindLinConjWithoutEdge($M,Y,R,i$)
\If {\chg{$U \geq 0$ \textbf{ and }}$Exist[U]=0$}
\State $Exist[U]$:= 1
\State push $U$ into $S(e(U))$
\EndIf
\EndIf
\EndFor
\State \chg{Print $R$}
\EndWhile
\EndFor
\EndProcedure
\end{algorithmic}
\end{algorithm}

\begin{prop} For any kinetic system and any suitable fixed set of complexes all the possible reaction graphs 
describing linearly conjugate realizations can be computed after finitely many steps by 
\emph{\textbf{Algorithm 1}}. The whole computation might last until exponential time depending on the number 
of different reaction graphs, but the time elapsed between the displaying of two linearly conjugate 
realizations is always polynomial.
\end{prop}

\begin{proof} 
Let us assume that there is a sequence $W$ which is not returned by the algorithm, but it describes a linearly 
conjugate realization of the kinetic system. Let $R$ be another realization, which was computed by the 
algorithm and $G_W$ is a subgraph of $G_R$, i.e. $R[i]=1$ if $W[i]=1$  for all $i \in \{1, \ldots ,N\}$. 
The sequence $D$ fulfils this property for each $W$, but if there is more than one such 
sequence, then let us chose $R$ to be the one with minimum number of coordinates equal to 1.
It follows from the definition of the sequences $W$ and $R$ that there must be an index $j$ so that 
$W[j]=0$ and $R[j]=1$. (If there is more than one such index, then let us chose the smallest one.)
During the computation there is a step (line 9) when we apply procedure 
\textbf{FindLinConjWithoutEdge}($M$,$Y$,$R$,$j$) to 
compute the sequence $U$, which is a dense realization with the properties: $U[j]=0$ and $G_U$ is a 
subgraph of $G_R$. \chg{According to the properties of $W$, it fulfils the constraints of the optimization problem 
specified in the procedure, therefore it cannot be infeasible and $G_W$ must be a subgraph of $G_U$.} 
But it leads to a contradiction, since if $W$ is equal to $U$, then $W$ is returned by the algorithm, and if 
they are not equal, then $R$ is not minimal.

A sequence $R$ popped out from stack $S(k)$ is \chg{printed out} after all its coordinates are 
examined. This computation requires the application of the procedure 
\textbf{FindLinConjWithoutEdge}($M$,$Y$,$R$,$i$)  $k$ times, with some additional minor computation. 
From the properties of the procedure it follows that the  computation between displaying two consecutive 
realizations can be performed in polynomial time.

\chg{In each stack there might be finitely many sequences (in stack $S(k)$ at most 
$\binom{N}{k}$) and in case of each sequence, only finitely many calls of \textbf{FindLinConjWithoutEdge} have to be performed, therefore the whole computation can be performed in finite time.}
\end{proof}
\chg{It is important to see that we compute the entries of the corresponding matrices $A_k$ and $T^{-1}$ in 
each step of Algorithm 1 according to Eqs. (12)-(16), but do not store them. This is possible because to set 
up the constraints for the subsequent computations, it is enough to know only the structure (i.e. the zero and 
non-zero entries) of matrix $A_k$ corresponding to the realization that is popped out of the actual stack. 
If one wants to store and/or display the actual reaction rate coefficients and transformation parameters, it 
is possible by slightly modifying the procedure \textbf{FindLinConjWithoutEdge} to return not only the 
resulting reaction graph encoded by the binary sequence $U$, but also the matrices $T$ and $A_k$.}

\subsection{Parallelization of the algorithm}
Due to the possible large number of different \chg{reaction graphs} even in the case of relatively small 
networks (see Example 2 in Section \ref{sec:examples}), it is worth studying the parallel implementation of 
the proposed algorithm. This may dramatically improve the computational performance as it has been shown in 
the case of several other fundamental problems (see, e.g. \cite{Wosniack2015,Nejad2015}).

For the sequences $R$ in each stack $S(k)$ and for all the indices $i \in \{1, \ldots,N \}$ the 
results coming from the executions of the procedure \textbf{FindLinConjWithoutEdge}($M,Y,R,i$) do not have any 
effect on each other, consequently these might be computed in a parallel way. The obtained sequences are 
pushed into stacks with indices smaller than $k$, but only the ones that have not been found earlier.
It follows that the procedure might be applied in parallel for sequences from different stacks as well, 
since the repetition of the same computation is avoided by the application of the array $Exist$.

In case of dynamically equivalent realizations it is possible to do even more parallelization. These are 
special linearly conjugate realizations, where the transformation matrix $T$ is the unit matrix, the variables 
are the entries of matrix $A_k$, and Equation \eqref{a} determining dynamical equivalence can be written 
in a more simple form:
\begin{equation}
Y \cdot A_k = M
\end{equation}
It is easy to see that the values of the variables in the $j$th column of matrix $A_k$ depend only on the 
parameters in the $j$th column of matrix $M$ \chg{and on the entries of matrix $Y$}. Therefore the columns of $A_k$ might be computed 
parallely, and any dynamically equivalent realization can be determined by choosing a possible solution 
in case of each column and build the Kirchhoff matrix of the realization from them. It follows that the number 
of dynamically equivalent realizations, that describe different reaction graphs, is the product of the numbers 
of the possible columns.

The super-structure property of dense realizations is inherited by the columns of the matrix $A_k$, 
therefore we can use the same algorithm for these as we used for linearly conjugate realizations.

In this algorithm it is better to determine the ordering of non-core edges according to columns. 
Let us denote  the number of non-core edges in column $j$ by $N_j$, and the sequence describing the $j$th  
column of the dense dynamically equivalent realization by $D_j$. The stacks are also needed to be defined 
separately for each column. The sequence $R_j$ representing a $j$th column gets stored in stack $S_j(k)$ if 
and only if the number of coordinates equal to $1$, denoted by $e(R_j)$, is exactly $k$.

We also need a two-dimensional binary array denoted by $ExistColumn$[$j,R_j$] to store the computed 
sequences in case of each column. The first index refers to the column and the second index is the sequence as 
a binary number. At the beginning all coordinates are equal to zero.

The applied procedures are as follows:

\begin{itemize}
\item  \textbf{DyneqColumnWithoutEdge}$(M,Y,j,R_j,i)$ computes the $j$th column of the Kirchhoff matrix 
describing a constrained dense dynamically equivalent realization of the kinetic system with coefficient 
matrix $M$ and complex composition matrix $Y$. The constraints are determined by the two last inputs, a 
sequence $R_j$ and an integer index $i$. The procedure returns a sequence $U_j$ representing a $j$th column so 
that $U_j[l]=0$ if $R_j[l]=0$ for all $l \in \{1, \ldots N_j\}$ and $U_j[i]=0$. \chg{If there is no such 
column, then $-1$ is returned.} This computation can be performed in polynomial time.

\item \textbf{Build$A_k$}($ExistColumn$) builds all possible  dynamically equivalent realizations from the 
sequence parts in $ExistColumn$ and saves them in the array $Exist$.
\end{itemize}

\begin{algorithm}[H]
\caption{Determines all reaction graphs describing dynamically equivalent realizations applying 
parallelization}
\begin{algorithmic}[1]

\Procedure{Dynamically equivalent graph structures}{$M,Y$}
\For {$j=1$ to $m$}
\State push $D_j$ into $S_j(N_j)$
\For {$k=N_j$ to $1$}
\While{size.$S_j(k)>0$}
\State $R_j$:= pop $S_j(k)$
\For {$i = 1$ to $N_j$}
\If {$R_j[i]=1$}
\State $U_j:=$ DyneqColumnWithoutEdge$(M,Y,j,R_j,i)$
\If {\chg{$U_j \geq 0$ \textbf{ and }}$ExistColumn[j,U_j]=0$}
\State $ExistColumn[j,U_j]$:=1
\State push $U_j$ into $S_j(e(U_j))$
\EndIf
\EndIf
\EndFor
\EndWhile
\EndFor
\EndFor
\State Build{$A_k$}($ExistColumn$)
\EndProcedure
\end{algorithmic}
\end{algorithm}

\section{Examples}\label{sec:examples}
In this section we demonstrate the operation of our algorithm on kinetic systems with a small  
(Section \ref{Ex_small}) and a slightly bigger (Section \ref{Ex_big}) set of complexes. The algorithm was 
implemented in MATLAB \cite{Ma:2000} using the YALMIP  modelling language \cite{Loefberg2004}. 

\chg{We will see that} the number of possible reaction graphs describing linearly conjugate realizations 
\chg{grows very fast} depending on the number of complexes.

\subsection{Example 1} \label{Ex_small}
In this example taken from \cite{Szederkenyi2011}, we examine the kinetic system described by the following 
dynamical equations:
\begin{align*}
\  &  \dot{x}_1 = 3k_1 \cdot x_2^3-k_2 \cdot x_1^3 \\
\  & \dot{x}_2 = -3k_1 \cdot x_2^3+k_2 \cdot x_1^3
\end{align*}
According to the monomials there are (at least) two species, $\mathcal{S}=\{X_1,X_2\}$, and we fix the set of 
complexes to be $\mathcal{C}=\{C_1,C_2, C_3\}$, where $C_1 = 3X_2$, $C_2 = 3X_1$, and $C_3 = 2X_1+X_2$. Based 
on the above, the inputs of the algorithm -- the matrices $Y$ and $M$ -- are as follows:
\begin{center}
$Y= \begin{bmatrix}
0 & 3 & 2 \\
3 & 0 & 1
\end{bmatrix} 
\qquad
M= \begin{bmatrix*}[r]
3k_1 & -k_2 & 0 \\
-3k_1 & k_2 & 0
\end{bmatrix*} $
\end{center}
For the numerical computations, the parameter values $k_1=1$ and $k_2=2$ were used. As the result of the 
algorithm we get 18 different sequences/reaction graphs. This small example is special in the sense that the 
sets of different reaction graphs corresponding to dynamically equivalent and linearly conjugate 
realizations are the same, since the computed transformation matrix $T$ was the unit matrix in each case. 
Using the numerical results, it was easy to symbolically solve the equations for dynamical equivalence, 
therefore we can give the computed reaction rate coefficients as functions of $k_1$ and $k_2$.  

The reaction graphs are denoted by $G_1, \ldots , G_{18}$ and are presented with 
suitable reaction rate coefficients in Figure \ref{fig:Ex_small}. From the computation it follows that there 
are two reaction rate coefficients $k_{31}$ and $k_{32}$ which do not depend on the input parameters, just on 
each other, and the reactions determined by these might together be \chg{present or} non-present in the 
reaction network. 
Therefore, a nonnegative parameter $p$ is applied to determine the values of these coefficients. We get 
the reaction graphs $G_1, \ldots , G_{9}$ if the parameter $p$ is positive, and if it is zero then we get the 
reaction graphs $G_{10}, \ldots , G_{18}$. The reaction graph $G_1$ (the complete directed graph) describes 
the dense realization, and consequently all other reaction graphs are subgraphs of it (not considering the 
edge weights).

\begin{figure}[H]
\begin{center}
    \includegraphics[width=1\textwidth]{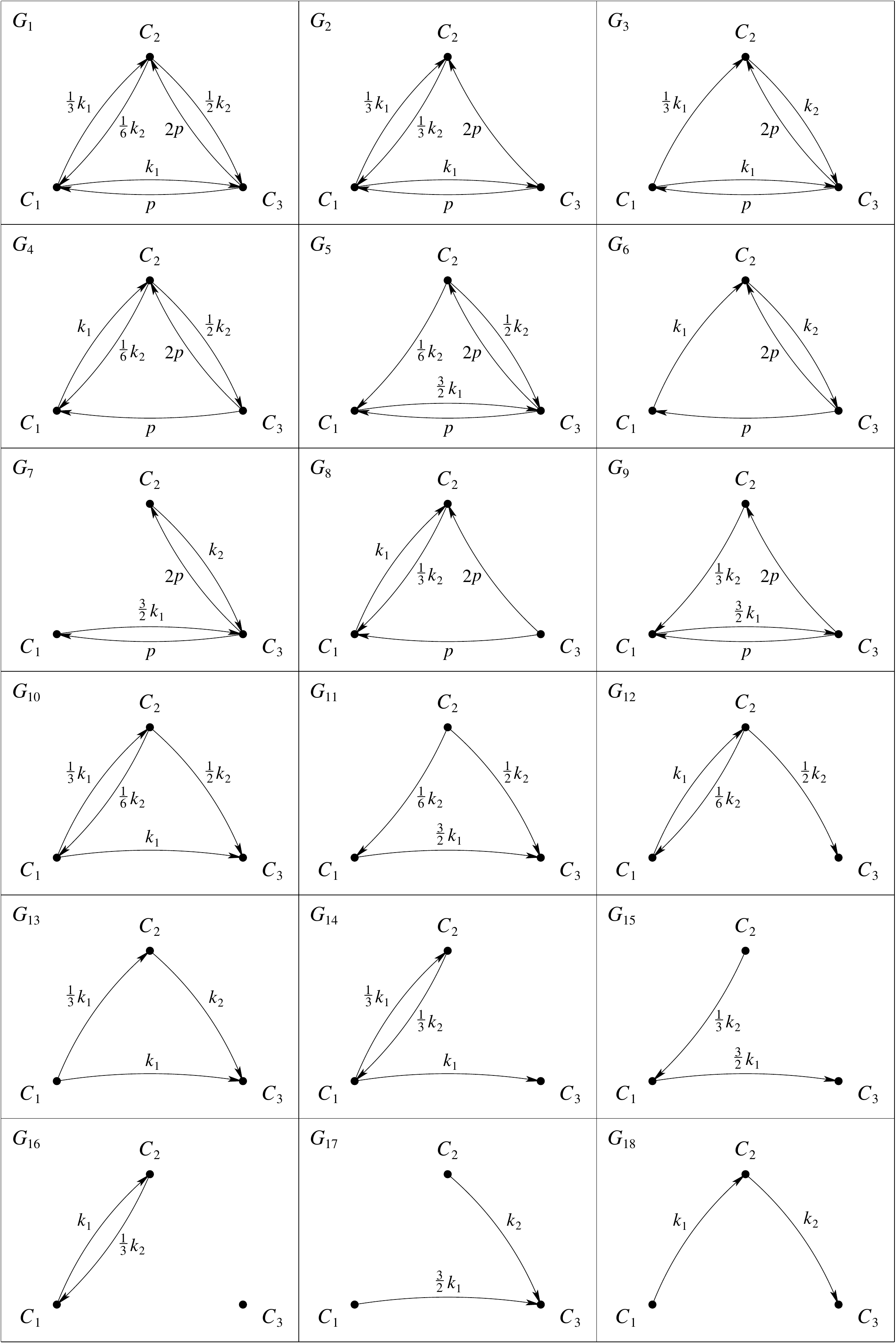}
\end{center}
  \caption{\chg{All reaction graphs of Example 1 with possible reaction rate coefficients}}
  \label{fig:Ex_small}
\end{figure}
\subsection{Example 2} \label{Ex_big}
The purpose of this example is to show the possible large number of structurally different linearly conjugate 
realizations even in the case of a relatively small kinetic system. 
The reaction network examined in this section was published in \cite{Csaszar1981} as example $A1$. In the 
original article it is given by the following realization, described by the Kirchhoff matrix $A_k$ and 
reaction graph shown in Figure \ref{fig:Ex_csaszar}.

\bigskip

\begin{multicols}{2}

\begin{figure}[H]
\begin{center}
    \includegraphics[width=0.26\textwidth]{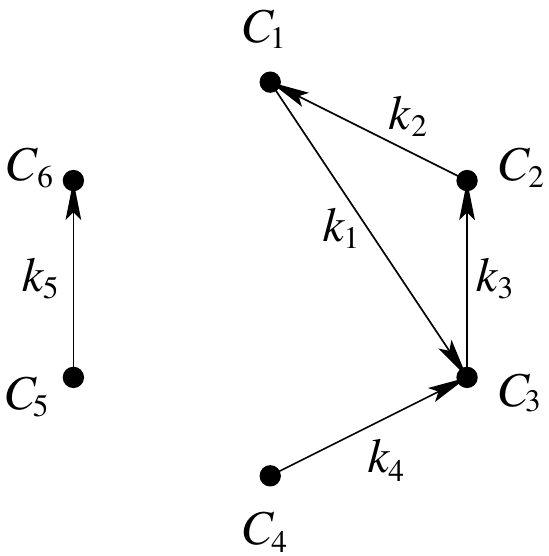}
\end{center}
  \caption{The reaction graph of Example 2}
  \label{fig:Ex_csaszar}
\end{figure}

\[A_k= \begin{bmatrix*}[r]
-k_1 & k_2 & 0 & 0 & 0 & 0 \\ 
0 & -k_2 & k_3 & 0 & 0 & 0 \\
k_1 & 0 & -k_3 & k_4 & 0 & 0 \\
0 & 0 & 0 & -k_4 & 0 & 0 \\
0 & 0 & 0 & 0 & -k_5 & 0 \\
0 & 0 & 0 & 0 & k_5 & 0 \\
\end{bmatrix*}\]

\end{multicols}

\noindent In the reaction network there are two species, $\mathcal{S}=\{X_1,X_2\}$ and six complexes,  
$\mathcal{C}=\{C_1=0,C_2=X_1,C_3=X_2,C_4=2X_1,C_5=2X_1+X_2,C_6=3X_1\}$. 
According to the definitions, the complex composition matrix $Y$ and the matrix $M= Y \cdot A_k$ of 
coefficients are as follows:

\[Y= \begin{bmatrix}
0 & 1 & 0 & 2 & 2 & 3 \\
0 & 0 & 1 & 0 & 1 & 0
\end{bmatrix} \qquad 
M= \begin{bmatrix}
0 & -k_2 & k_3 & -2k_4 & k_5 & 0 \\
k_1 & 0 & -k_3 & k_4 & -k_5 & 0
\end{bmatrix}\] 

\noindent The reaction rate coefficients used in the computations were the same as in \cite{Csaszar1981}, 
namely: $k_1=1,~k_2=1,~k_3=0.05,~k_4=0.1,~k_5=0.1$. With these parameter values the system shows oscillatory 
behaviour.
In the dense \chg{linearly conjugate} realization shown in Figure 
\ref{fig:Ex_csaszar_dense}, there are 19 reactions, and it can be 
given by the matrices $A_k^d$ and $T^d$ as:

\begin{multicols}{2}

\begin{figure}[H]
\begin{center}
    \includegraphics[width=0.3\textwidth]{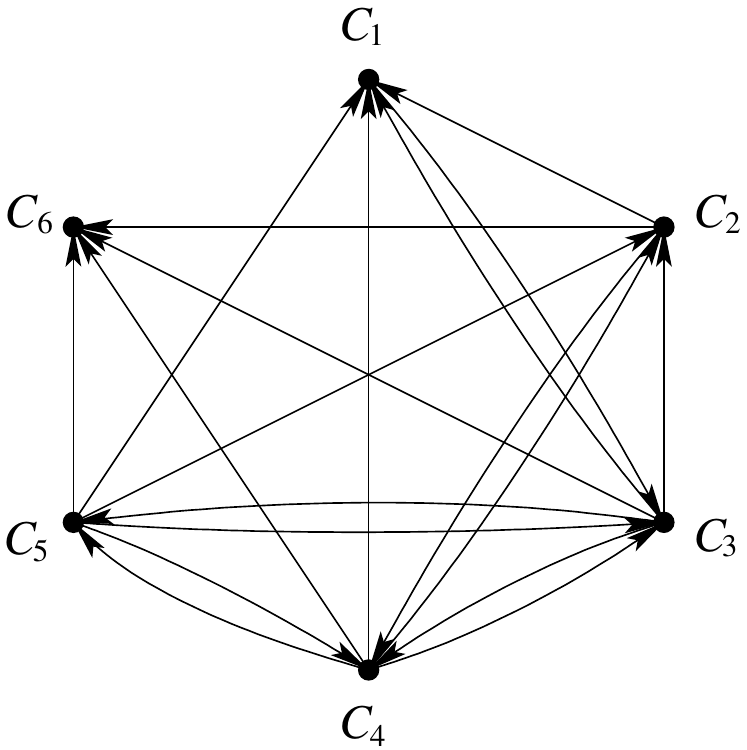}
\end{center}
  \caption{The reaction graph of the dense linearly conjugate realization}
  \label{fig:Ex_csaszar_dense}
\end{figure}

\begin{footnotesize}
\[ A_k^d= \begin{bmatrix*}[r]
-80 & 1.167 \cdot 10^7 & 3.083 & 3.333 \cdot 10^6 & 0.333 & 0 \\ 
0 & -2 \cdot 10^7  & 0.5 & 5 \cdot 10^6 & 0.5 & 0 \\
80 & 0 & -4.25 & 4 & 0.5 & 0 \\
0 & 5 \cdot 10^6 & 0.25 & -2 \cdot 10^7 & 1 & 0 \\
0 & 0 & 0.25 & 4 & -8.5 & 0 \\
0 & 3.333 \cdot 10^6 & 0.167 & 1.167 \cdot 10^7 & 6.167 & 0 \\
\end{bmatrix*}\]

\[(T^d)^{-1}= \begin{bmatrix*}[r]
40 & 0\\ 
0 & 80
\end{bmatrix*}\]
\end{footnotesize}
\end{multicols}
Our algorithm returned as many as 17160 different reaction graphs describing linearly conjugate realizations 
of this kinetic system, all of which can be found in the electronic supplement available at: 
\begin{center}
\mbox{\texttt{http://daedalus.scl.sztaki.hu/PCRG/works/publications/Ex2\_AllRealSuppl.pdf}} 
\end{center}
Out of these, 17154 can be described by a weakly connected reaction graph, while 6 have disconnected reaction 
graphs, with the same linkage classes. Since this property can be ensured by linear constraints 
(the edges between the linkage classes are excluded), according to Proposition \ref{superstr} the realization 
having the maximum number of edges determines a super-structure among realizations obeying the 
same constraints. This constrained dense realization can be 
described by the matrices $A_k^{ld}$ and $T^{ld}$, and 
its reaction graph is shown in Figure \ref{fig:Ex_csaszar_2lc}. 
\begin{multicols}{2}

\begin{figure}[H]
\begin{center}
    \includegraphics[width=0.3\textwidth]{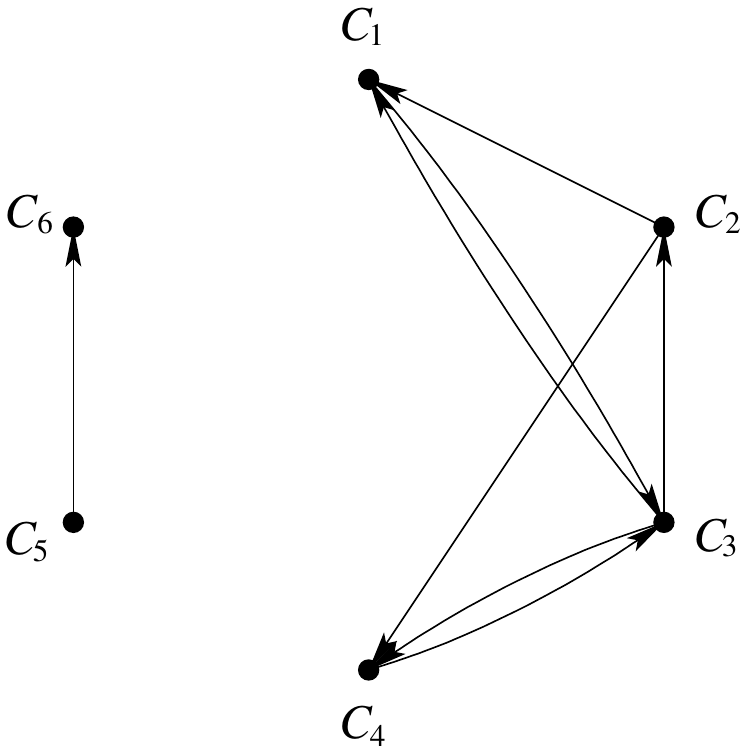}
\end{center}
  \caption{The reaction graph of the dense linearly conjugate realization having two linkage classes}
  \label{fig:Ex_csaszar_2lc}
\end{figure}

\footnotesize
\[ A_k^{ld}= \begin{bmatrix*}[r]
-50 & 1.25 \cdot 10^7 & 0.625 & 0 & 0 & 0 \\ 
0 & -2.5 \cdot 10^7  & 1.25 & 0 & 0 & 0 \\
50 & 0 & -2.5 & 5 & 0 & 0 \\
0 & 1.25 \cdot 10^7 & 0.625 & -5 & 0 & 0 \\
0 & 0 & 0 & 0 & -5 & 0 \\
0 & 0 & 0 & 0 & 5 & 0 \\
\end{bmatrix*}\]

\[(T^{ld})^{-1}= \begin{bmatrix*}[r]
50 & 0\\ 
0 & 50
\end{bmatrix*}\]
\end{multicols}
It also turned out from the computations that in this case the sparse realization is unique, and it is the 
initial network shown in Figure \ref{fig:Ex_csaszar}. The distribution of the computed reaction graphs over the number of reactions is shown in Fig. \ref{fig:edge_distr_A1}. 
\begin{figure}[H]
\begin{center}
    \includegraphics[trim=2cm 3cm 2cm 3cm, clip=true, width=0.8\textwidth]{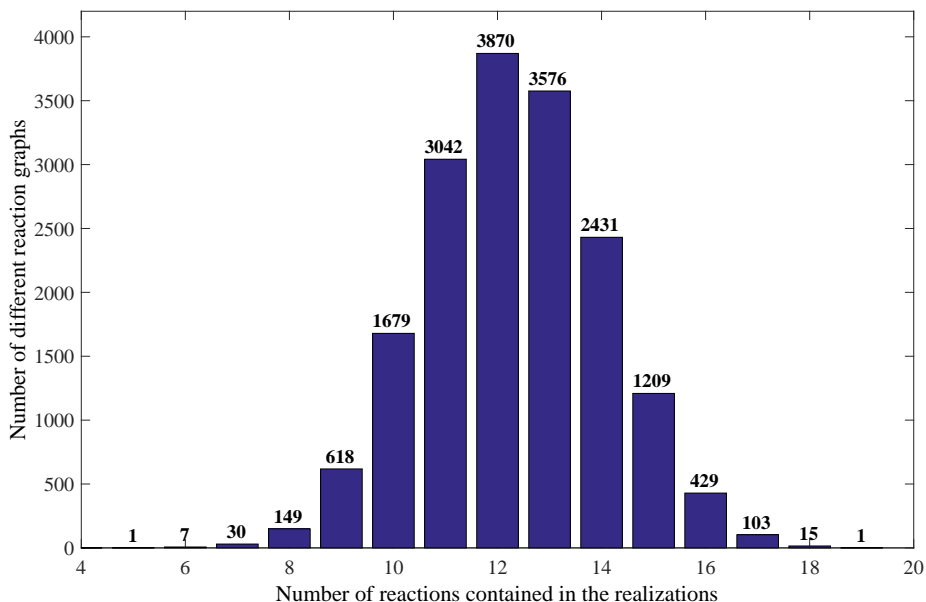}%
\end{center}
  \caption{\chg{Number of different reaction graphs with given numbers of reactions (directed edges) in the case of Example 2}}
  \label{fig:edge_distr_A1}
\end{figure}

Figures \ref{fig:A1_simulation_orig} and \ref{fig:A1_simulation_linconj} show the solutions of the original 
and the dense linearly conjugate realizations from consistent initial conditions. The oscillatory behaviour is 
clearly shown, and one can check from the results that $\overline{x}(t) = (T^d)^{-1}\cdot x(t)$ holds for all 
$t\geq 0$.

\begin{figure}[H]
\begin{center}
    \includegraphics[trim=6cm 5cm 6cm 5cm, clip=true, width=0.6\textwidth]{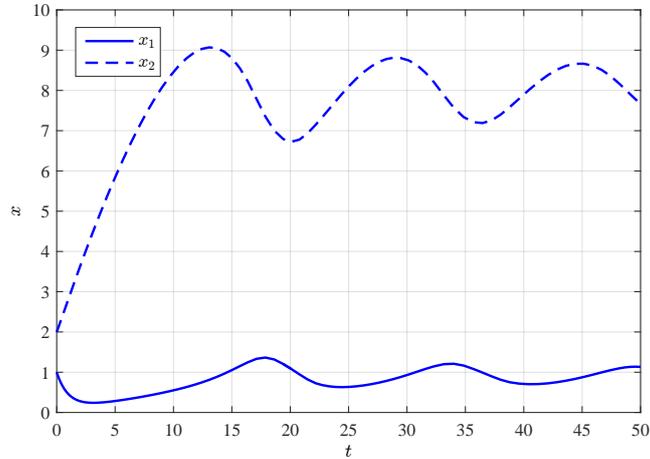}%
\end{center}    
  \caption{\chg{Solution of the original system in Example 2 defined by the matrix pair $(Y,A_k)$ from the initial conditions $x(0)=[1~2]^T$}}
  \label{fig:A1_simulation_orig}
\end{figure}
\begin{figure}[H]
\begin{center}
    \includegraphics[trim=6cm 5cm 6cm 5cm, clip=true, width=0.6\textwidth]{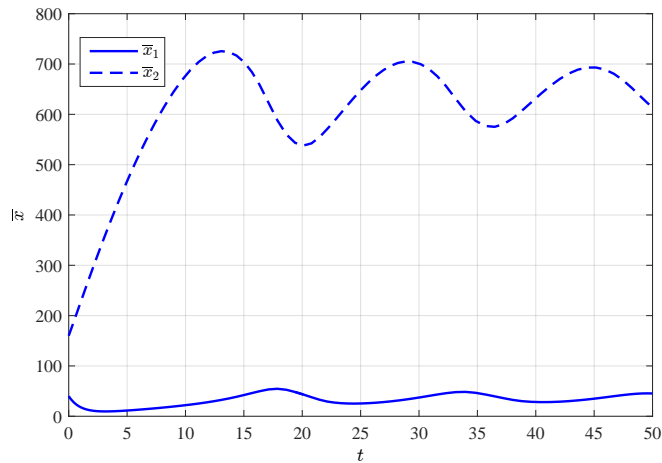}%
\end{center}    
  \caption{\chg{Solution of the linearly conjugate system in Example 2 defined by the matrix pair $(Y,A_k^d)$ from   the initial conditions $\overline{x}(0)=(T^d)^{-1}\cdot x(0) = [40~160]^T$}}
  \label{fig:A1_simulation_linconj}
  \end{figure}
  
\subsection{\chg{Computation results of the parallel implementation}}

\chg{We tested the parallel implementation of Algorithm 1 on a workstation
with two 2.60GHz Xeon (E5-2650 v2) processors with 32 Gb RAM (DDR3 1600 MHz, 0.6ns). The implementation 
was written in Matlab R2013a using the built-in parallelization toolbox.}

\chg{Let $L$ denote the number of threads. The tests were carried out with $L \in \{$1, 2, 4, 6, 8, 10, 12$\}$, where these 
threads are working on the innermost loop (lines 7 to 15) of the algorithm. We can easily conclude that the 
most time-consuming step is the calling of procedure \textbf{FindLinConjWithoutEdge} which required on average 1.9 ms 
computation time with 0.6 ms standard deviation in case of Example 2 (calculated considering 130 000 executions). }
\chg{Figure \ref{fig:total_time} shows the total execution times 
in case of both examples with different numbers of threads on log-log scales.}

\begin{figure}[H]
\begin{center}
    \includegraphics[trim=6cm 5cm 6cm 5cm, clip=true, width=0.7\textwidth]{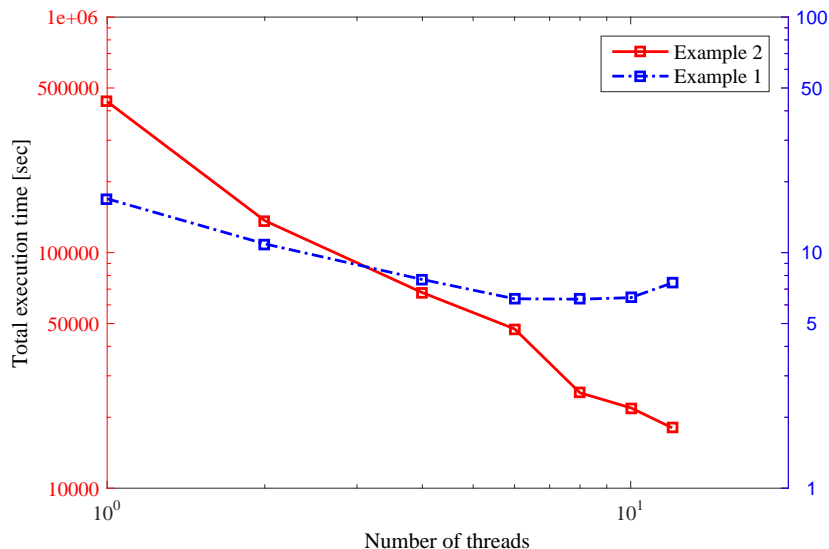}%
\end{center}    
  \caption{\chg{Overall execution times of Examples 1 and 2 with different numbers of threads}}
  \label{fig:total_time}
\end{figure}

\chg{One can expect that the computation of $N$ different edge exclusions by using the procedure 
\textbf{FindLinConjWithoutEdge} reaches maximum efficiency when $L=N$ holds.
It can be seen that in the case of Example 1 (where $N=6$) we reach the maximum efficiency indeed, and for a larger number of 
threads we get slightly longer computation times.
The improving efficiency with the larger number of threads can be seen in the case of Example 2, where the value of $N$ was 19. The results show in the studied thread-range that by doubling the number of threads, the total execution time approximately gets halved.}

\chg{In case of Example 2 we have recorded the computation times separately for the different stacks as well. Here the computation time means the time elapsed from popping out the first element from a given stack until the end of processing the last element in the same stack.  The obtained results are shown in Table \ref{tabl_stacks}. Since the sparse realization contains five reactions, the procedure \textbf{FindLinConjWithoutEdge} was not called for sequences with fewer than five nonzero coordinates. This is why stacks with index smaller than 5 are not included in the table.
}

\begin{table}
\begin{small}
\begin{center}
\begin{tabular}{|c|c|ccccccc|}
\hline
\  & \ & \multicolumn{7}{c|}{Number of threads}\\
\hline
\  & \  & 1 & 2 & 4 &  6 & 8 & 10 & 12\\
\hline
\  & $S(5)$ &       17.667    &   5.7767    &   2.6061   &    1.9417    &   1.0794   &   0.91359  &    0.83738 \\
\  & $S(6)$ &       149.11    &   42.734    &   19.173   &    16.935    &    8.763   &    7.3073  &     6.5739 \\
\  & $S(7)$ &       762.26    &   234.79    &   87.411   &    100.92    &   42.237   &    35.511  &     31.838 \\
\  & $S(8)$ &       4458.7    &   1233.5    &   482.81   &    532.03    &   236.23   &    197.91  &      177.8 \\
\  & $S(9)$ &        20355    &   6247.5    &   2503.8   &    2356.1    &   1278.1   &     891.4  &     794.43 \\
\  & $S(10)$ &        59492    &    17246    &   6921.6   &    6769.3    &     3511   &    2536.5  &     2248.6 \\
\  & $S(11)$ &   1.0534e+05    &    33706    &    13570   &     11661    &   5982.5   &    5338.1  &     3878.8 \\
Stacks  & $S(12)$ &   1.1675e+05    &    36972    &    20753   &     11958    &   6507.2   &    5825.6  &     4253.1 \\
\  & $S(13)$ &        83201    &    25604    &    14972   &    8653.4    &   4572.2   &    4085.2  &     3770.1 \\
\  & $S(14)$ &        35556    &    10750    &   5960.5   &    3755.5    &   2126.5   &    1958.5  &     1858.1 \\
\  & $S(15)$ &       9283.8    &   3501.6    &   1737.8   &    1197.8    &   787.78   &    749.61  &     723.94 \\
\  & $S(16)$ &         2015    &   898.08    &   478.54   &    353.27    &   259.49   &    246.49  &     234.34 \\
\  & $S(17)$ &        397.5    &   205.36    &   125.82   &    81.447    &   75.907   &    57.191  &     55.652 \\
\  & $S(18)$ &       57.236    &   29.231    &   17.953   &    11.173    &   11.078   &    8.6908  &     8.0421 \\
\  & $S(19)$ &       4.9291    &   3.3659    &   2.7989   &    2.0673    &   1.9249   &    1.8916  &     1.7575 \\
\hline
\end{tabular}
\end{center}
\end{small}
\caption{\chg{Processing times (in seconds) of the individual stacks in case of Example 2 for different numbers of threads}}\label{tabl_stacks}
\end{table}

\section{Conclusions}
An algorithm was proposed in this paper for computing all \chg{structurally different reaction graphs 
describing} linearly conjugate realizations of a given kinetic polynomial system. To the best of the authors' 
knowledge this method is the first provably correct solution for the exhaustive search of CRN structures 
realizing a given dynamics. The inputs of the algorithm are the complex composition matrix and the coefficient 
matrix of the studied polynomial system. The output is the set of all possible reaction graphs encoded by binary 
sequences. The correctness of the method is proved using a recent result saying that the linearly constrained 
dense realization determines a super-structure among all realizations fulfilling the same constraints 
\cite{Acs2015}. 
Although exponentially many different reaction graphs may exist, it is shown that polynomial time is elapsed 
between displaying (storing) two consecutive reaction graphs. The computation starts with the determination of 
the dense realization and different stacks are maintained for storing realizations with the same number of 
reactions. The number of stacks depends linearly on the number of reactions in the dense realization, although 
the entire bookkeeping may require exponential storage space due to the possible large number of different 
structures. This organization allows that the optimization tasks for processing the actual realization (i.e. 
computing its constrained immediate `successors') can be implemented parallely. The applicability of the algorithm is illustrated on two examples taken from the literature. \chg{The numerical results show that parallel implementation indeed improves efficacy of the computation.}

\section*{Acknowledgements}
This project was developed within the PhD program of the Roska Tamás Doctoral School of Sciences and 
Technology, Faculty of Information Technology and Bionics, Pázmány Péter Catholic University, Budapest. 
The authors gratefully acknowledge the support of the Hungarian National  Research,  Development  and  Innovation  Office  --  NKFIH through grants OTKA NF104706 and 115694. The support of Pázmány Péter Catholic University is also acknowledged through the project KAP15-052-1.1-ITK. The authors thank Dr. István Reguly for his help in evaluating the computation results of the parallel implementation of the algorithm. 

\end{document}